\documentclass[12pt]{article}
\usepackage{amsmath,amssymb,amsthm,amsfonts}
\usepackage{hyperref}
\usepackage{tikz}
\usepackage{enumerate}
\makeatletter
\def\imod#1{\allowbreak\mkern10mu({\operator@font mod}\,\,#1)}
\makeatother

\newtheorem{theorem}{Theorem}[section]

\newtheorem{lemma}{Lemma}[section]
\newtheorem{corollary}{Corollary}[section]

\theoremstyle{definition}
\newtheorem{definition}{Definition}[section]

\allowdisplaybreaks

\usepackage{lscape}
\usepackage{graphicx}
\usepackage{placeins}
\usepackage{float}
\usepackage{subfigure}
\usepackage{color}
\usepackage{wrapfig}  
\usepackage{float}   
\usepackage{lastpage}  
\usepackage{tkz-graph}
\usepackage{tikz-qtree}

\usepackage{graphicx}
	\graphicspath{{./figures/}}
\title{A New Lower Bound for van der Waerden Numbers\footnote{\textcopyright This manuscript version is made available under the CC-BY-NC-ND 4.0 license \indent\indent http://creativecommons.org/licenses/by-nc-nd/4.0/}}

\author{Thomas Blankenship\thanks{Dept.\ of Mathematics and Statistics, Sacramento State University, {\tt thomasblankensh@csus.edu}} \and 
Jay Cummings\thanks{Dept.\ of Mathematics and Statistics, Sacramento State University, {\tt jay.cummings@csus.edu}.} \and 
Vladislav Taranchuk\thanks{Dept.\ of Mathematics and Statistics, Sacramento State University, {\tt vtaranchuk@csus.edu}} \and 
}

\begin{document}
\maketitle

\begin{abstract}
In this paper we prove a new recurrence relation on the van der Waerden numbers, $w(r,k)$. In particular, if $p$ is a prime and $p\leq k$ then $w(r, k) > p \cdot \left(w\left(r - \left\lceil \frac{r}{p}\right\rceil, k\right) -1\right)$. This recurrence gives the lower bound $w(r, p+1) > p^{r-1}2^p$ when $r \leq p$, which generalizes Berlekamp's theorem on 2-colorings, and gives the best known bound for a large interval of $r$.  The recurrence can also be used to construct explicit valid colorings, and it improves known lower bounds on small van der Waerden numbers.
\end{abstract}


\section{Introduction and History}
In 1927, van der Waerden proved that for any positive integers $r$ and $k$ there exists an $N = w(r,k)$ such that every $r$-coloring of $\{1,2,3,\dots,N\}$ contains a monochromatic arithmetic progression of length $k$.  As a central function in Ramsey theory and a notoriously difficult one to understand, the growth rate of $w(r,k)$ has received much attention.

Van der Waerden's initial proof gives a monstrous upper bound.  In the slowest-growing case, when $r=2$, still the bound is $w(2,k) \leq A(n)$, where $A(n)$ is the Ackerman function.  The best known general upper bound is due to Gowers \cite{Gowers}, who proved
$$
w(r, k) \leq 2^{2^{r^{2^{2^{k+9}}}}}.
$$
In \cite{GrahamSol} Graham and Solymosi improved this in the case when $k=3$, 
which in a series of follow-up papers by Bourgain \cite{Bourgain}, Sanders \cite{sanders} and Bloom \cite{bloom} further improved the upper bound to
$$w(r,3) \leq 2^{cr(\ln r)^4}$$
where $c > 0$ is an absolute constant.  Graham currently offers 1,000 USD for an answer as to whether or not $w(2,k) < 2^{k^2}$.

In 1953, Erd\H{o}s and and Rado \cite{Erdos} proved the lower bound
$$\sqrt{2(k-1)r^{k-1}} \leq w(r,k)$$
using a simple counting argument.  In 1960, Moser \cite{Moser} used a constructive approach to improve this bound in the case that $r$ is large relative to $k$.  In particular, he showed that
$$(k-1)r^{C\ln(r)} < w(r,k)$$
for some absolute constant $C$.  Two years later Schmidt \cite{Schmidt} used a nonconstructive approach to prove a bound that is asymptotically better in $k$.  He showed that that there is some absolute constant $c$ for which
$$r^{k - c\sqrt{k \ln(k)}} \leq w(r,k).$$

In 1968, Berlekamp used an algebraic approach to construct what is still the best known lower bound for the case when $k=p+1$, where $p$ is a prime, and $r=2$.  He showed that
$$p2^p < w(2,p+1).$$

In this paper we use a construction to generalize this result to the following.

\noindent{\bf Theorem \ref{mainthemsbound}.} \ 
If $p$ is any prime with $2 \leq r \leq p \leq k$, then
$$
p^{r-1}2^p < w(r, p+1).
$$
This generalizes Berlekamp's theorem \cite{BKamp}.



In 1973, Erd\H{o}s and Lov\'{a}sz \cite{ErdosLovasz} used the Lov\'{a}sz Local Lemma on hypergraphs to show that
$$\frac{r^{k-1}}{4k}\left(1 - \frac{1}{k}\right) \leq w(r,k).$$

In this paper we will use a recurrence to generalize Berlekamp's result to arbitrary number of colors.  Our work will also improve bounds on small van der Waerden numbers.  Finally, our bound is recursively-constructive, in that an explicit coloring when $r=2$ can be used create explicit colorings for larger $r$.

The current best known general lower bound of this type is due to Kozik and Shabanov \cite{shab}, who in 2016 proved
$$c \cdot r^{k-1} \leq w(r,k)$$
for some absolute constant $c > 0$.

For large $r \gg k$, the best result, by O'Bryant, can be obtained by using the Hypergraph Symmetry Theorem and the Behrend-type results about sets of integers without long progressions (see \cite{Obryant}):
$$w(r, k) >e^{f(k) (\ln r)^{\lceil\log_2k\rceil}}$$
where $f(k)$ is a function of k. The above bound can be found in \cite{brown} and is best known for large $r \gg k$.

There are now constructive approaches to the Lov\'{a}sz Local Lemma (see \cite{Gasarch}), which can be used to produce explicit constructions with high probability.  Therefore, in a sense, the above two bounds can also be considered constructive.

%
%
%
%
%

In the following section we establish a recursive lower bound for $w(r, k)$, which is used to deduce our main result.  In Section 4 we use this recurrence relation to improve known numerical lower bounds for some small values of $r$ and $k$.

\section{Proof of the Main Theorem}

\begin{definition}
Let $R_r$ represent the set of colors $\{ 1, 2, \dots, r \}$. For each $i \in R_r$, define $S_i(r, k)$ to be the $p$-tuple
$$S_i(r,k) = (i,i+1,i+2,\dots,r,1,2,3,\dots,r,1,2,\dots),$$
where $p$ is the largest prime such that $p \leq k$.
\end{definition}

For example,
\begin{align*}
S_1(5,11) &= (1,2,3,4,5,1,2,3,4,5,1)\\
S_2(5,11) &= (2,3,4,5,1,2,3,4,5,1,2)\\
S_3(5,11) &= (3,4,5,1,2,3,4,5,1,2,3)\\
S_4(5,11) &= (4,5,1,2,3,4,5,1,2,3,4)\\
S_5(5,11) &= (5,1,2,3,4,5,1,2,3,4,5).
\end{align*}

Note that, among these, only $S_1(5,11)$ has the property that $i$ is in a position congruent to $i$ (mod 5).

\begin{definition}  Here are some basic definitions.
\begin{itemize}
\item  A color $c$'s \emph{index positions} with a tuple $S_i(r,k)$ are the set of its positions within that tuple.
\item  The tuples $S_i(r,k)$ are sometimes called a \emph{block}.
\item  The acronym \emph{$k$-TMAP} is short for a $k$-term monochromatic arithmetic progression.
\item  If $x_1,x_2,\dots,x_k$ is an arithmetic progression, then the common value $x_{i+1}-x_i$ is called the \emph{common difference}.
\item  For a fixed $k$ and $r$, a \emph{valid} $r$-coloring of $\{1,2,3,\dots,n\}$ is one which contains no $k$-TMAP.
\end{itemize}
\end{definition}

Note that if you know the index positions of a color $c \in R_r$, then you can determine which of the $r$ distinct tuples $S_i(r, k)$ it is in.  For instance, if $c$ is in index positions congruent to 2 (mod $r$), then it must be in the tuple $S_{c-1}(r,k)$.  This works in general because within a fixed tuple all index positions for $c$ fall into the same residue class (mod $r$),
and among the $r$ different tuples $S_i(r, k)$, the color $c$ falls into each residue class exactly once.

\begin{lemma}\label{UnionSi}
Fix an $r$ and $k$ and let $p$ be the largest prime such that $p \leq k$.  We can choose $(r - \lceil \frac{r}{p} \rceil)$ distinct $S_i(r, k)$'s such that the concatenation $T$ of any ordering of these tuples has the property that any $k$-TMAP in $T$ has common difference divisible by $p$.
\end{lemma}

\begin{proof}
This proof is best understood when split into two cases.


Case 1:  Assume $r \leq p$.  
Given any collection of the blocks
$$S_2(r,k), S_3(r,k), \dots, S_r(r,k),$$
each included with any multiplicity, let $T$ be the concatenation of any ordering of these blocks.  Suppose $(x_1,x_2, \dots, x_k)$ is a $k$-TMAP in $T$ of color $c$. If some $x_i$ and $x_j$ have difference
$x_j - x_i = d \equiv 0(\text{mod } p)$, then by the primality of $p$ we may conclude that the $k$-TMAP has common difference divisible by $p$, and hence we are done.

Therefore, we may now assume the common difference is not congruent to 0 (mod $p$), which implies that the index positions of $(x_1, x_2, \dots, x_k)$ must be distinct (mod $p$), and hence span every possible index position.  
However,  for this to be possible, some element would have to be in an index position congruent to $c$ (mod $p$), meaning it is inside of the tuple $S_1(r,k)$.  But since we initially excluded this tuple, this is a contradiction.

Case 2:  Assume $r > p$.  The construction of $T$ from Case 1 is no longer sufficient; by attempting to use that construction we could still conclude that a $k$-TMAP $(x_1,x_2, \dots, x_k)$ of some color $c$ and with common difference $d \not\equiv 0\ (\text{mod } p)$ must contain index positions from $p$ different residue classes (mod $p$).  However, since $p<r$, each tuple no longer contains every color.  Therefore removing one tuple does not guarantee that we have removed the color $c$, and if we did not remove this color then $c$ still appears in index positions in all residue classes.  If we remove $\lceil\frac{r}{p}\rceil$ block types, though, we can again guarantee that each color's index positions appear in at most $p-1$ conjugacy classes, preventing a monochromatic progression in that color.  Indeed, if we consider a concatenated tuple $T$ of all $S_i(r, k)$'s except for those of the form 
$S_{1+pm}(r, k)$ where $m \in \{ 0, 1, \dots , \lceil \frac{r}{p} \rceil -1 \}$, then this leaves behind ($r - \lceil \frac{r}{p} \rceil$) $S_i(r, k)$'s to build $T$ from, and with this the same argument follows as for Case 1.

%
%
\end{proof}

\begin{theorem}\label{MainThm}
We have the recurrence that, for $r \geq 2$,
$$w(r, k) > p\left(w\left(r - \left\lceil \frac{r}{p} \right\rceil, k\right) -1\right),$$
where $p$ is the largest prime such that $p \leq k$.
\end{theorem}

\begin{proof}
By the definition of $w(r - \lceil \frac{r}{p} \rceil, k)$, there exists an $(r - \lceil \frac{r}{p} \rceil)$-coloring of $\{ 1, 2, \dots , w(r - \lceil \frac{r}{p} \rceil, k) -1 \}$ containing no $k$-TMAP.  We now ``blow up'' this coloring by replacing each color $i$ with the block $S_i(r, k)$; that is, we now have a coloring of
$$\left\{ 1, 2, \dots , p \cdot \left(w\left(r - \left\lceil \frac{r}{p}\right\rceil, k\right) -1\right)\right\},$$
where if before $c_i$ was the color of $i$, then now the colors of the integers
$$pi,pi+1,\dots,pi + p - 1$$
match the values of $S_{c_i}(r,k)$.  

We now invoke Lemma \ref{UnionSi} to see that the only $k$-TMAP that can exist in this coloring are those with common difference $d \equiv 0$ (mod $p$).  Suppose such an arithmetic progression $(x_1,x_2,\dots,x_k)$ has common difference $mp$ and is in color $c$.  Then we can see that any two terms of this arithmetic progression must be the same color and in the same index position as each other, in their respective blocks; in particular, this implies that all blocks containing an $x_i$ must be of the same block-type $S_{\ell}(r,k)$.  Consider now the progression $(x_1',x_2',\dots,x_k')$ in the original valid coloring of
$$\left\{ 1, 2, \dots , w\left(r - \left\lceil \frac{r}{p}\right\rceil, k\right) - 1\right\},$$
where $x_i'=t$ if $x_i$ was in the $t^{\text{th}}$ block in the blown up coloring.  Notice that $(x_1',x_2',\dots,x_k')$ is an arithmetic progression with common difference $m$, and is moreover monochromatic of color $\ell$.

This is a contradiction to our assumption that we originally chose an $(r - \lceil \frac{r}{p} \rceil)$-coloring of $\{ 1, 2, \dots , w(r - \lceil \frac{r}{p} \rceil, k) -1 \}$ containing no $k$-TMAP.
\end{proof}

Notice that when $p > r$ that each $S_i(r, k)$ contains at least $\lfloor \frac{r}{p} \rfloor$ elements of each color. This means that when we remove one of these block types, we remove at least $\lfloor \frac{r}{p} \rfloor$ index positions, and so we have no more than $p - \lfloor \frac{r}{p} \rfloor$ possible index positions that any color might be in. We can therefore make the recursion of Theorem \ref{MainThm} stronger by choosing a larger prime $p$ such that $p - \lfloor \frac{r}{p} \rfloor < k$. This means that when we have blocks of this larger size $p$ and remove one block type, we are still left with fewer than $k$ possible index positions for elements to be in, ensuring that the proof of Theorem \ref{MainThm} still holds.

\begin{corollary}\label{Construction}

Fix $k$ and let $p$ be the largest prime less than or equal to $k$.  For $r \in \{1,2,3,4,\dots,p\}$ we can recursively construct $r$-colorings of size 
$$
p^{r-1}(k-1)
$$
containing no $k$-TMAPs.  
\end{corollary}

\begin{proof}
Begin with $k-1$ consecutive monochromatically colored elements, and then replace each element with an $S_i(2, k)$ tuple. Since our original coloring did not contain a $k$-TMAP, Lemma \ref{UnionSi} tells us that our new construction will likewise contain no $k$-TMAPs. We then likewise replace each of the two colors in the new construction with their corresponding $S_i(3, k)$ tuples. Continue replacing elements with the next tuple until you have replaced elements with $S_i(r, k)$ tuples. We know that the size of every tuple is $p$ and we recursively replace elements  with tuples from the constructed colorings $r-1$ times.
\end{proof}

\begin{theorem}\label{mainthemsbound}
If $p$ is any prime with $2 \leq r \leq p \leq k$, then
$$
w(r, p+1) > p^{r-1}2^p.
$$
This generalizes Berlekamp's theorem \cite{BKamp}.
\end{theorem}

\begin{proof}
Berlekamp's result states that $w(2, p+1) > p2^p$. This implies that there is a valid 2-coloring of $\left\{1,2,3,\dots,p2^p\right\}$ that contains no $k$-TMAPs. By invoking Corollary \ref{MainThm}, we can use Berlekamp's result as a base case for constructing valid $r$-colorings for larger $r$.
\end{proof}

\section{Comparing to Previously Known Bounds}

Note that any known valid 2-coloring can be used to create valid $r$-colorings by using the construction method found in the proof of Corollary \ref{Construction}.  When $r > p$ the recurrence in Theorem \ref{MainThm} gives worse and worse bounds: When $r \in \{1,2,\dots,p\}$ we get a new factor of $p$ each time we increment $r$; when $r \in \{p+1,p+2,\dots,2p\}$ we only get a new factor of $p$ every other time; in general, when $r \in \{\ell\cdot p + 1, \ell\cdot p+2,\dots,(\ell+1)\cdot p\}$, we on average get a new factor of $p$ every one $\ell^{\text{th}}$ of the time.

Indeed, the growth rate of the number of factors of $p$ can be seen to be proportional to that of harmonic numbers.  In particular, if one writes $r = \ell\cdot p + s$ where $s \in \{0,1,2,\dots,p-1\}$ and $H_{\ell} = \sum_{i=1}^{\ell} \frac{1}{i}$, then one can show that, for fixed $k$ and $p$ and $r \to\infty$,
$$w(r,k) \gtrapprox p^{p\cdot H_{\ell}+\frac{s}{\ell +1}}(p2^p) \sim c\cdot p^{p\ln\frac{r}{p}}2^p,$$
for some constant $c$.  Moreover $w(r,k) > p^{p\ln\frac{r}{p}}2^p$.  Note that when $r \geq k$ the bound in Theorem \ref{MainThm} beats out Kozik and Shabanov's bound \cite{shab},  which was the best known when $r$ and $k$ are similar in size.



Ours stops being best when O'Bryant and Moser-type bounds take over.

\section{Bounds on Small Van Der Waerden Numbers}

There are only seven known van der Waerden numbers, showing just how difficult these problems are, both theoretically and computationally.  Indeed, the last two numbers found, that $w(6,2) = 1,132$ \cite{Kouril} and $w(4,3) = 76$ \cite{beeler}, were discovered using SAT solvers and computers specifically designed for this task.  Another extensive computational effort used the Berkeley Open Infrastructure For Network Computing to distribute the work to 516 volunteers' 1,760 computers in 53 countries, totaling two teraflops of computing power for a full year.  \cite{Monroe}

In Figure \ref{smallvdwnumbers} we list the best known lower bounds for small values of $k$ and $r$.  The numbers in bold are new or have been improved by Theorem \ref{MainThm}.

\begin{figure}[h]
\scalebox{0.75}{
$\begin{array}{ | l | l | l | l | l | l | l | l | l | l | l}
\hline
	k / r & 2 \text{ Colors} & 3 \text{ Colors} & 4 \text{ Colors} & 5 \text{ Colors} & 6 \text{ Colors}\\ \hline
	3-\text{Term} & 9 & 27 & 76 & >170 &\bf  >225  \\ \hline
	4-\text{Term} & 35 & 293 & >1,048 & >2,254 & >9,778 \\ \hline
	5-\text{Term} & 178 & >2,173 & >17,705 & >98,740 & >98,748  \\ \hline
	6-\text{Term} & 1132 & >11,191 & >91,331 & >540,025 & >816,981  \\ \hline
	7-\text{Term} & >3,703 & >48,811 & >420,217 & >\bf 2,941,519 &\bf  >20,590,633  \\ \hline
	8-\text{Term} & >11,495 & >238,400 & >2,388,317 & >\bf 16,718,219 &\bf  >117,027,533  \\ \hline
	9-\text{Term} & >41,265 & >932,745 & >10,898,729 & >79,706,009 &\bf  >557,942,063   \\ \hline
	10-\text{Term} & >103,474 & >4,173,724 & >76,049,218 & >542,694,970  &\bf  >3,798,864,790  \\ \hline
	11-\text{Term} & >193,941 & >18,603,731 & >329,263,781 & >\bf 3,621,901,591 &\bf  >39,840,917,501  \\ \hline
	12-\text{Term} & >638,727 & >79,134,144 & >1,536,435,264 & >\bf 16,900,787,904 &\bf  >185,908,666,944  \\ \hline
	13-\text{Term} & >1,642,309 & >251,282,317 & >5,683,410,589 & >\bf 73,884,37,657 &\bf  >960,496,389,541 \\ \hline
\end{array}$
}

\vspace{.5cm}

\scalebox{0.75}{
$\begin{array}{ | l | l | l | l | l | l | l | l | l | l | l}
\hline
	k / r & 7 \text{ Colors} & 8 \text{ Colors} & 9 \text{ Colors} \\ \hline
	3-\text{Term} &\bf  >225 &\bf  >510 &\bf  >775  \\ \hline
	4-\text{Term} &\bf  >9,940 &\bf  >29,334 &\bf  >29,334  \\ \hline
	5-\text{Term} &\bf  >493,700 &\bf  >493,740 &\bf  >2,468,500  \\ \hline
	6-\text{Term} &\bf  >2,700,125 &\bf  >4,084,905 &\bf  >13,500,625  \\ \hline
	7-\text{Term} &\bf  >144,134,431 &\bf  >144,134,431 &\bf  >1,008,941,017  \\ \hline
	8-\text{Term} &\bf  >819,192,732 &\bf  >819,192,732 &\bf  >5,734,349,124  \\ \hline
	9-\text{Term} &\bf  >3,905,594,441 &\bf  >3,905,594,441 &\bf  >27,339,161,087  \\ \hline
	10-\text{Term} &\bf  >26,592,053,530 &\bf  >26,592,053,530 &\bf  >186,144,374,710  \\ \hline
	11-\text{Term} &\bf  >438,250,092,511 &\bf  >4,820,751,017,621 &\bf  >53,028,261,193,831  \\ \hline
	12-\text{Term} &\bf  >2,044,995,336,384 &\bf  >22,494,948,700,224 &\bf  >247,444,435,703,464   \\ \hline
	13-\text{Term} &\bf  >12,486,453,064,033 &\bf  >162,323,889,832,429 & \bf >2,110,210,567,821,577 \\ \hline
\end{array}$
}
\caption{Small van der Waerden numbers}\label{smallvdwnumbers}
\end{figure}

\newpage

\section{Concluding Remarks}

In the proof of Theorem \ref{MainThm}, our construction of the $S_i(r, k)$ had the property that no two consecutive elements are ever the same color. This implies that in any coloring that is built from the $S_i(r, k)$ tuples, there can only ever be two consecutive terms that are the same color, which would occur only between the ends of the consecutive blocks. Thus our construction may be applicable to Ramsey problems which demand that a coloring avoids many consecutive monochromatic colored terms, such as the problem introduced by Graham in \cite{Graham}.

\section{Acknowledgments}

The authors would like to thank Craig Timmons for his helpful comments which improved this paper.  We would also like to thank the referees whose helpful comments improved the quality of this paper.




\end{document}